\documentclass[a4,12pt]{article}
\usepackage{amssymb,amsmath,lineno}
\newtheorem{theorem}{Theorem}[section]

\newtheorem{prop}[theorem]{Proposition}
\newtheorem{cor}[theorem]{Corollary}
\newcommand{\qed}{\qquad$\Box$}
\newtheorem{unnumber}{}

\newtheorem{prepf}[unnumber]{Proof}
\newenvironment{proof}{\prepf\rm}{\endprepf}

\def\gsg{\Gamma_{gsg}}
\def\psg{\Gamma_{psg}}

\newcommand{\gr}{\mathop{\mathrm{girth}}}
\begin{document}
\title{Subgroup sum graphs of finite abelian groups}
\author{Peter J. Cameron\footnote{School of Mathematics and Statistics, University of St Andrews, North Haugh St Andrews, Fife, KY16 9SS, U.K., Email:~pjc20@st-andrews.ac.uk, ORCID:0000-0003-3130-9505},~
R. Raveendra Prathap\footnote{Research Scholar, Department of Mathematics, Manonmaniam Sundaranar University, Tirunelveli 627012, Tamil Nadu, India, Email:~rasu.raveendraprathap@gmail.com}\\
and T. Tamizh Chelvam\footnote{Department of Mathematics, Manonmaniam Sundaranar University, Tirunelveli 627012, Tamil Nadu, India, Email:~tamche59@gmail.com, ORCID:0000-0002-1878-7847}}
\date{}
\maketitle

\begin{abstract}
Let $G$ be a finite abelian group, written additively, and $H$ a subgroup of~$G$. The \emph{subgroup sum graph} $\Gamma_{G,H}$ is the graph with vertex set $G$, in which two distinct vertices $x$ and $y$ are joined if $x+y\in H\setminus\{0\}$. These graphs form a fairly large class of Cayley sum graphs. Among cases which have been considered previously are the \emph{prime sum graphs}, in the case where $H=pG$ for some prime number $p$. In this paper we present their structure and a detailed analysis of their properties. We also consider the simpler graph $\Gamma^+_{G,H}$, which we refer to as the \emph{extended subgroup sum graph}, in which $x$ and $y$ are joined if $x+y\in H$: the subgroup sum is obtained by removing from this graph the partial matching of edges having the form $\{x,-x\}$ when $2x\ne0$. We study perfectness, clique number and independence number, connectedness, diameter, spectrum, and domination number of these graphs and their complements. We interpret our general results in detail in the prime sum graphs.
\end{abstract}

\section{Introduction}
Cayley graphs are excellent models for interconnection networks. Hence, there are many investigations in connection with parallel processing and distributed computing. The definition of the Cayley graph was introduced by Arthur Cayley in $1878$ to explain the concept of abstract groups which are described by a set of generators.  Cayley graphs of finite cyclic groups are studied in the name of circulant graphs~\cite{MT, QH, MH, TTSR1, TTSRIG} and Cayley graphs of finite  groups are considered in ~\cite{MK, DEJ03, LAK99, LEE01, TTGKWC, TTSMU, TTSMU2, TTMS4}. Other graphs from finite groups are also studied in \cite{CS, IS, AV1, TM}.   Several authors studied Cayley graphs of finite abelian groups in~\cite{EP, WT, TTIR4, EY}.  The  generalized Cayley graphs of finite rings with respect to subsets are studied in~\cite{T20196, T20194}. 
 
The square element graph of rings was studied by Biswas, Sen Gupta and Sen~\cite{BR, RM, SS}, while the power graph of semigroups was studied in~\cite{RK}. The power graph of groups are studied through the orders of elements in a group in~\cite{CS, IS, AV1, TM}. Raveendra Prathap and Tamizh Chelvam~\cite{RPTT19, RPTT20} defined and studied about the square graph and cubic power graph of finite abelian groups. Let $G$ be a finite abelian group with identity element $0.$ The square  graph of $G$ denoted $\Gamma_{sq}(G)$ is an undirected simple graph with vertex set $G$ and two distinct vertices $a$ and $b$ are adjacent in $\Gamma_{sq}(G)$ if $a+b=2t$ for some $t\in G$ and $2t\neq 0.$ Having defined the square graph $\Gamma_{sq}(G)$ of $G,$ authors studied various properties of the complement of square graph in~\cite{RPTT19}. Subsequently another graph called the cubic power graph $\Gamma_{cpg}(G)$ is introduced and studied in~ \cite{RPTT20}. These graphs can be generalized, in the context of Cayley graphs of finite abelian groups, in parallel with the generalizations made in the case of finite rings~\cite{T20196, T20194}. For a fixed positive integer $n$, the generalized sum graph $\gsg(G)$ is the simple undirected graph with vertex set $G$ and two distinct vertices $x$ and $y$ are adjacent if $x+y\in S=\{nt\mid nt\neq 0, t\in G\}.$  One can see that when $n=2, \gsg(G)=\Gamma_{sq}(G)$ and when $n=3, \gsg(G)=\Gamma_{cpg}(G).$ Further note that when $n=1$ and  $S$ is a generating set for $G,$ $\gsg(G)$ is the Cayley graph Cay$(G,S)$~\cite{LAK99}.  When $n$ is a prime number $p,$ we call the generalized sum graph as the prime sum graph $\psg(G).$

In this paper, we extend the process of generalization, by defining the \textit{subset sum graph} $\Gamma_{G,H}$, where $G$ is a finite abelian group and $H$ a subgroup of $G$: the vertices are the elements of $G$, and $x$ and $y$ are joined if $x+y\in H\setminus\{0\}$. A closely related graph, which we call the \textit{extended subset sum graph} $\Gamma^+_{G,H}$ is defined similarly, but without the restriction $x+y\ne0$ for adjacency; it turns out to be easier to work with.

\section{Preliminaries} 
In this section, we recollect certain basic definitions and properties of graphs which are essential for further reference.  Throughout this paper, $\Gamma=(V,E)$ is a finite simple graph with vertex set $V$ and edge set $E.$  A graph $\Gamma$ is said to be \textit{connected} if there exists a path between every pair of distinct vertices in $\Gamma.$  A graph $\Gamma$ is said to be \textit{complete} if every pair of distinct vertices are adjacent through an edge and the complete graph on $n$ vertices is denoted by $K_n.$ The \textit{degree} of a vertex $v$ is the number of the edges in $\Gamma$ which are incident with $v.$ Note that degree of each vertex $v$ in $K_n$ is $n-1.$  The distance $d(u,v)$ between the vertices $u$ and $v$ in $\Gamma$ is the length of the shortest path between $u$ and $v.$ If no path exists between $u$ and $v$ in $\Gamma,$ then $d(u,v) = \infty.$ For a vertex $v\in V(\Gamma),$ the \textit{eccentricity~ $e(v)$} of $v$ is the maximum distance from $v$ to any other vertex in $V(\Gamma).$ That is, $e(v)=\max\{d(v,w):w \in V(\Gamma)\}.$ The \textit{radius} of $\Gamma$ is the minimum eccentricity among the vertices of $\Gamma$ and is denoted by $rad(\Gamma).$  i.e., $rad(\Gamma)=\min\{e(v):v \in V(G)\}.$ The \textit{diameter} of $\Gamma$ is the maximum eccentricity among the vertices of $\Gamma$ and is denoted by $diam(\Gamma).$ i.e., $diam(\Gamma)=\max\{e(v):v \in V(G)\}.$ The \textit{girth} of $\Gamma$ is the length of a shortest cycle in $\Gamma$ and is denoted by $\gr(\Gamma).$ 

% A graph $\Gamma$ is said to be a \textit{self-centred} graph if the eccentricity of every vertex in $\Gamma$ is same. In other words, a graph is a self-centred graph if radius and diameter of the graph are equal.  

% A graph $\Gamma$  with $n$ vertices is called a \textit{refinement of a star graph} if $K_{1,n-1}$ is a subgraph of $\Gamma$ and $\Gamma$ is not complete. Every star graph is a refinement of a star graph but converse need not be true.
A \textit{clique} of $\Gamma$ is a maximal complete subgraph of $\Gamma$ and the number of vertices in the largest clique of $\Gamma$ is called the \textit{clique number} of $\Gamma$ and is denoted by $\omega(\Gamma)$.
%A graph $\Gamma$ with $n$ vertices is called \textit{pancyclic} if it contains a cycle of length $k$ for every $3\leq k \leq n,$ and it is called \textit{vertex pancyclic} if every vertex is contained in a cycle of length $k$ for every $3\leq k \leq n.$

For a vertex $x\in V(G),$ $N(x)$ is the set of all vertices in $G$ which are adjacent to $x$ and $N[x]= N(x)\cup\{x\}.$ An \textit{independent set} is a set of vertices in a graph $\Gamma$, in which no two vertices are adjacent. The cardinality of a maximal independent set is called the \textit{independence number} and is denoted by $\beta(\Gamma)$. A (vertex) \textit{proper colouring} of $\Gamma$ is an assignment of colours from a set $C$ such that no two adjacent vertices receive same colour. If $|C| = k,$ we say that the corresponding colouring is a proper $k$-colouring. A graph is \textit{$k$-colourable} if it has a proper $k$-colouring. The \textit{chromatic number} of a graph $\Gamma$ is the least $k$ such that $\Gamma$ is $k$-colourable and is denoted by $\chi(\Gamma)$. The \textit{clique
cover number} $\theta(\Gamma)$ is the smallest number of complete subgraphs required to cover all the vertices of $\Gamma$. Note that the independence number and clique cover number of $\Gamma$ are just the clique number and chromatic number of the complementary graph $\overline{\Gamma}$.

A graph $\Gamma$ is \textit{perfect} if every induced subgraph of $\Gamma$ has clique number equal to chromatic number. The \textit{Weak Perfect Graph Theorem} of Lov\'asz~\cite{Lovasz} asserts that the complement of a perfect graph is also perfect; so every induced subgraph of a perfect graph has independence number equal to clique cover number. We also make use of the theorem of Dilworth~\cite{Dilworth} asserting that the comparability graph (or incomparability graph) of a partial order is perfect.

The \textit{open neighbourhood} $N_\Gamma(v)$ of the vertex $v$ in $\Gamma$ is the set of vertices adjacent to $v$, while the \textit{closed neighbourhood} of $v$ is $\{v\}\cup N_{\Gamma}(v)$. The \textit{domination number} of a graph is the least cardinality of a set of vertices for which the union of their closed neighbourhoods is the whole vertex set 

\section{Definition and basic properties}
In this section, we give formal definitions of our graphs, describe their structure in terms of the structure of $G$ and $H$, examine connectedness, diameter, girth, and self-centredness, and show that these graphs are perfect. Let $G$ be a finite abelian group, and $H$ a subgroup of $G$. We define the
\emph{extended subgroup sum graph} $\Gamma^+_{G,H}$ to have vertex set $G$ and edges $\{x,y\}$ whenever $x+y\in H$; and the \emph{subgroup sum graph} $\Gamma_{G,H}$ to have the same vertex set and edges $\{x,y\}$ whenever $x+y\in H\setminus\{0\}$.

We see that the generalized sum graph $\Gamma_{gsg}(G)$ previously mentioned is the subgroup sum graph $\Gamma_{G,tG}$, while for any prime $p$, the prime sum graph $\Gamma_{psg}(G)$ is the subgroup sum graph $\Gamma_{G,pG}$.

\medskip

The next result deals with the case where the subgroup $H$ is trivial (either $\{0\}$ or $G$).

\begin{theorem}\label{t:trivial} Let $G$ be a finite abelian group.
\begin{enumerate}
\item If $H=\{0\}$, then the subgroup sum graph $\Gamma_{G,H}$ is a null  graph on the vertex set $G$, while the extended subgroup sum graph $\Gamma^+_{G,H}$ is a partial matching where elements other than the identity and involutions are joined to their inverses. 
\item If $H=G$, then the extended subgroup sum graph $\Gamma^+_{G,H}$ is complete, and the subgroup sum graph $\Gamma_{G,H}$ is obtained by deleting a matching covering all vertices except the identity and involutions.
\end{enumerate}
\end{theorem}

\begin{proof}
If $H=\{0\},$ then the only edges in $\Gamma^+_{G,H}$ are those of the form $\{a,-a\}$ where $2a\neq 0$; there are no edges in $\Gamma_{G,H}$.

If $H=G,$ then every pair $\{a,b\}$ with $a\ne b$ is an edge of $\Gamma^+_{G,G}$, and all those with $b\ne -a$ are edges of $\Gamma_{G,H}$.\qed
\end{proof}

The graphs considered in Theorem~\ref{t:trivial} are not very  interesting, so where necessary below we assume that $1<|H|<|G|$.

For a prime $p$ and an abelian group $G,$ we have the following for the prime sum graph $\Gamma_{psg}(G).$
\begin{enumerate}
\item $pG=\{0\}$ if and only if $G$ is elementary abelian (a direct sum of cyclic groups $C_p$ of order $p$).
\item $pG=G$ if and only if $p$ does not divide $|G|$.
\end{enumerate}

Let $S(G)$ denote the set of solutions of $2x=0$ in a group $G$, and $s(G)=|S(G)|$. If $G=C_{2^{k_1}}\times\cdots\times C_{2^{k_r}}\times A$, where $|A|$ is odd and $k_1,\ldots,k_r>0$, then $s(G)=2^r$.
\medskip

The basic structure of these graphs $\Gamma^+_{G,H}$ and $\Gamma_{G,H}$ are given in the next result.

\begin{theorem}\label{t:main}
Let $H$ be a subgroup of the abelian group $G$, with $|H|=k$ and $|G/H|=m$.
\begin{enumerate}
\item The extended subgroup sum graph $\Gamma^+_{G,H}$ has $(m+s(G/H))/2$ connected components, of which $s(G/H)$ are complete graphs $K_k$ (whose vertex sets are the cosets of $H$ having order $1$ or $2$ in $G/H$), and $(m-s(G/H))/2$ are complete bipartite graphs $K_{k,k}$ (whose vertex sets are
the union of two cosets $H+a$ and $H-a$ for some $a\in G$ with $2a\notin H$). 
\item The subgroup sum graph $\Gamma_{G,H}$ is obtained from the extended subgroup sum graph $\Gamma^+_{G,H}$ by deleting a perfect matching from every component which is complete bipartite, and deleting a matching from a complete component on a coset $H+a$ covering all elements other than elements of $S(G)$ lying in this coset (if any).
\end{enumerate}
\end{theorem}

\begin{proof}
(a) The neighbours of $a$ in $\Gamma^+_{G,H}$ are the elements of the coset $H-a$ (possibly excluding $a$). If $H-a\ne H+a$, then we have a complete bipartite graph on these two cosets. If $H-a=H+a$, so that this coset has order $2$ in $G/H$, then we have a complete graph on this coset.
\medskip

(b) To obtain the subgroup sum graph, we must delete edges of the form $\{a,-a\}$ for which $a\ne -a$ (that is, $a$ is not the identity or an involution). 
\qed
\end{proof}

Theorem~\ref{t:main} gives a complete description of graphs $\Gamma^+_{G,H}$ and $\Gamma_{G,H},$ and enables us to determine their properties and parameters, as we do in the rest of the paper. First, though, we describe the components in a little more detail, and introduce three parameters we will use throughout the paper, counting three
different types of cosets of $H$ in $G$:
\begin{itemize}
\item \emph{Type 1}, cosets $H+a$ for which $2a\notin H$ (that is, cosets distinct from their inverses in $G/H$). For such cosets, $(H+a)\cup(H-a)$ is a connected component of both $\Gamma^+_{G,H}$ and $\Gamma_{G,H}$, being complete bipartite in the first and complete bipartite minus a perfect matching in the second.
\item \emph{Type 2}, cosets $H+a$ for which $2a\in H$ but $H+a$ does not contain a solution of $2x=0$. For these, $H+a$ is a connected component of both graphs, and is complete in the first and complete minus a perfect matching in the second. (This type can only occur if $|H|$ is even. We have to exclude
$|H|=2$ here since in that case $K_2$ minus a matching consists of two isolated vertices.)
\item \emph{Type 3}, cosets $H+a$ containing a solution of $2x=0$. For these,  $H+a$ is a connected component, and is complete in the first graph and complete minus a matching of size $(k-s(H))/2$ in the second. (This case always occurs, since the coset $H$ has this form. See below for the proof
that every such coset contains $s(H)$ solutions of $2x=0$.)
\end{itemize}

For $i=1,2,3$, we let $m_i$ be the number of cosets of Type~$i$, so that $m_1+m_2+m_3=m=|G/H|$, and $m_2+m_3=s(G/H)$.

We claim that the number of solutions of $2x=0$ in a coset is equal to $s(H)$ if the coset has Type~3, $0$ otherwise. This is clear for the coset $H$, so let $H+a$ be another coset. Clearly there are no solutions of $2x=0$ in the coset unless it has Type~3, so suppose this is the case. Then $K=H\cup(H+a)$ is a subgroup of $G$, and $S(K)$ is a subgroup of $G$ containing $S(H)$ as a subgroup of index $2$. Thus $s(K)=2s(H)$, so there are $s(H)$ elements in $H+a$ satisfying $2x=0$. In particular, we see that $m_3s(H)=s(G)$.

Thus, in $\Gamma^+_{G,H}$, there are $m_1/2$ components which are complete bipartite $K_{k,k}$, and $m_2+m_3$ components which are complete graphs $K_k$. In $\Gamma_{G,H}$, there are $m_1/2$ components which are a complete bipartite graph minus a perfect matching, $m_2$ components which are complete graphs $K_k$ minus a perfect matching, and $m_3$ components which are $K_k$ minus a matching of size
$(k-s(H))/2$.

In particular, we see that the numbers $m$, $s(G)$, $s(H)$ and $s(G/H)$ determine the number of cosets of each type.

\section{Connectedness, diameter, girth and perfectness}
We consider first a group of graph properties and parameters.

\begin{prop}\label{p:conn} Let $G$ be an abelian group and $H$ be a subgroup of $G.$
Suppose that $|G|>2$. Then the following are equivalent:
\begin{enumerate}
\item $\Gamma^+_{G,H}$ is connected;
\item $\Gamma_{G,H}$ is connected;
\item $H=G$.
\end{enumerate}
\end{prop}

\begin{proof}
The number of components of $\Gamma^+_{G,H}$ is $(m+s(G/H))/2$. Since each of 
$m$ and $s(G/H)$ is at least $1$, the graph is connected if and only if both 
are $1$, which implies that $H=G$. The converse is clear from
Theorem~\ref{t:trivial}.

Since $\Gamma_{G,H}$ has at least as many components as $\Gamma^+_{G,H}$,
we see that if it is connected then $H=G$. Conversely, if $H=G$, then we could
only disconnect the graph by deleting a matching if $|G|=2$. \qed
\end{proof}

\begin{cor}
Suppose that $|G|>2$ and $H\ne G$. Then the complements of $\Gamma^+_{G,H}$
and $\Gamma_{G,H}$ are connected and have diameter at most~$2$, with equality
except for $\overline{\Gamma}_{G,\{0\}}$ (which is complete). If $k>2$, then
these complements have radius equal to diameter, and so are self-centred.
\end{cor}

\begin{proof}
It is clear that the diameter of complements of $\Gamma^+_{G,H}$ and $\Gamma_{G,H}$ is $2$. If a vertex $g$ has eccentricity~$1,$ then it is joined to all other vertices, that is, it is isolated in
$\Gamma^+_{G,H}$ or $\Gamma_{G,H}$ as appropriate. This can only occur if $k=|H|=2$. 
\qed
\end{proof}

\begin{cor}
\begin{enumerate}
\item The extended subgroup sum graph $\Gamma^+_{G,H}$ has girth $3$ if
$|H|>2$, $4$ if $|H|=2$ and $G/H$ is not an elementary abelian $2$-group, or
$\infty$ otherwise.
\item The subgroup sum graph $\Gamma(G,H)$ has girth $3$ if $|H|>3$,
$6$ if $|H|=3$, or $\infty$ if $|H|=2$.
\end{enumerate}
\end{cor}

\begin{proof}
(a) There is always at least one component $K_k$, since $s(G/H)\ge1$; if
$k>2,$ this component contains a triangle. Otherwise, there is a component
$K_{k,k}$ unless $s(G/H)=m$.

\medskip

(b) If we delete a matching which is not perfect from a complete graph on at least four vertices, what is left contains a triangle. If $k=3,$ then $s(G)=s(G/H)$ and so there are no cycles of length $3$, but removing a perfect matching from $K_{3,3}$ produces a $6$-cycle. If $k=2,$ then the graph consists
of isolated vertices and edges.\qed
\end{proof}

\begin{cor}\label{c:perfect}
The subgroup sum graph and extended subgroup sum graph are perfect graphs.
\end{cor}

\begin{proof}
It suffices to show that each connected component is perfect. For the
extended graph, these components are complete or complete bipartite, which
are well known to be perfect. In the other case, we have to deal with the
complement or bipartite complement of a matching. In the first case, the
result holds since a matching is clearly perfect. The graph obtained by
removing a perfect matching from $K_{k,k}$ is the comparability graph of the
partial order on $\{a_1,\ldots,a_k,b_1,\ldots,b_k\}$ in which
$\{a_1,\ldots,a_k\}$ and $\{b_1,\ldots,b_k\}$ are antichains and $a_i<b_j$
if and only if $i\ne j$; and the comparability graph of a poset is perfect,
by Dilworth's Theorem.
\qed
\end{proof}

From the above theorem, we have the following corollary for the cubic power graph of a finite abelian group.
\begin{cor}{\normalfont \cite[Theorem 4.11]{RPTT20}}\label{cubic:perfect} Let $G$ be a finite abelian group. Then the cubic power graph $\Gamma_{cpg}(G)$ is perfect.
\end{cor}
Also, we have the following corollary which is applicable for the complement of the square graph of a finite abelian group.
\begin{cor}{\normalfont \cite[Theorem 3.2]{RPTT19}}\label{csquare:perfect}
Let $G$ be a finite abelian group. Then the complement $\overline{\Gamma}_{sq}(G)$ is perfect.
\end{cor}

\section{Cliques and cocliques}
We now compute the clique number and independence number of subgroup sum graphs $\Gamma_{G,H}$ where $H$ is a proper subgroup of $G.$

\subsection{General results}
We compute the clique number and independence number of the graphs $\Gamma_{G,H}$ where $H$ is a proper subgroup of $G.$ (By Corollary~\ref{c:perfect}, the chromatic number is equal to the clique number, and the clique cover number is equal to the independence number, so we get these two further parameters for free.)

\begin{theorem}\label{thm5.1}
Let $G$ be a finite abelian group, and $H$ a non-trivial subgroup of $G$.
Suppose that $|H|=k$, $|G/H|=m$, and let $m_1,m_2,m_3$ be as above.
\begin{enumerate}
\item The clique number of the extended subgroup sum graph $\Gamma^+_{G,H}$
is equal to $k$, and the independence number is equal to $km_1/2+m_2+m_3$.
\item The clique number of the subgroup sum graph $\Gamma_{G,H}$ is $(k+s(H))/2$.
\item The independence number of the subgroup sum graph $\Gamma_{G,H}$ is equal to
$km_1/2+2(m_2+m_3)$ if $s(H)<|H|$, and $km_1/2+2m_2+m_3$ if $s(H)=|H|$.
\end{enumerate}
\end{theorem}

\begin{proof}
(a) Clearly $H$ is a maximal clique, while a maximal independent set takes one
bipartite block from each complete bipartite component and one vertex from
each complete component. (By assumption, $k\ge2$, so the cliques of size $2$
in the complete bipartite graphs are not larger than $k$.)

\medskip

(b) When deleting edges from the extended subgroup sum graph, the bipartite
components remain bipartite, and so have clique number $2$.

In a complete component, corresponding to a coset of Type~2, we delete a
perfect matching, giving a graph with clique number $k/2$. If there is a coset
of Type~3, it contains a clique of size $(k+s(H))/2$. The second quantity is
larger, and there is always a Type~3 coset, namely $H$ itself. Now since
$k\ge 2$, we have $(k+s(H))/2\ge 2$.

\medskip

(c) It is easy to see that if we delete a perfect matching from $K_{k,k}$ with
$k\ge2$, the resulting graph still has independence number $k$.

A Type~2 coset carries a complete graph with a perfect matching removed, so
contains an independent set of size~$2$. Similarly, a Type~3 coset contains
an independent set of size~$2$ unless all its elements satisfy $2x=0$, in
which case it is complete and has independence number~$1$ (this occurs if and
only if $s(H)=|H|$). 

Summing over all cosets gives the result.
\qed
\end{proof}
Using the above Theorem~\ref{thm5.1}, one can obtain the clique number of the cubic power graph obtained in~\cite[Theorem 3.8]{RPTT20}. In fact, in the language of the cubic power graph $\Gamma_{cpg}(G)$ of an abelian group $G,$ $H=3G$ and $|H|=k=\frac{|G|}{3^r}.$ Hence, we have the following corollary.
\begin{cor}{\normalfont \cite[Theorem 3.8]{RPTT20}}\label{chro:cubic} Let $G$ be a finite abelian group. Then the chromatic number $\omega(\Gamma_{cpg}(G)=s(H)+\left\lceil\frac{\frac{|G|}{3^r}-s(H)}{2}\right\rceil=(k+s(H))/2.$
\end{cor}

\subsection{Prime sum graphs}
Now we calculate these numbers for the prime sum graph of a finite abelian group. It is clear from our analysis that the prime $2$ behaves very differently from odd primes. We will deal with odd primes first. Note that, by our earlier assumptions, we can assume that $|G|$ is divisible by the prime $p$ in question, otherwise $pG=G$.

\subsubsection{Odd prime $p$}
Let $G$ be a finite abelian group. Then $H=pG$ is a subgroup of $G$ whose index is $p^r$, where $r$ is the number of cyclic summands of $p$-power order when $G$ is expressed as a direct sum of cyclic groups; the quotient $G/pG$ is an elementary abelian $p$-group. So $s(G/H)=1$. Also, all $2$-elements
of $G$ are contained in $pG$, so (in the notation used above) we have $m_2=0$, $m_3=1$, and $m_1=p^r-1$. Also, if $G$ has $q$ cyclic summands of $2$-power order, then $s(H)=2^q$, and $s(H)=|H|$ if and only if $G\cong C_2^q\times C_p^r$.

From our earlier results, we find that $\Gamma^+_{G,pG}$ has clique number
$k=|G|/p^r$ and independence number $k(p^r-1)/2+1$, while $\Gamma_{G,pG}$
has clique number $k/2+2^{q-1}$ and independence number
$k(p^r-1)/2+2$ unless $G\cong C_2^q\times C_p^r$, in which case the 
independence number is $k(p^r-1)/2+1$. (If $q=0$, then $|G|$ is odd, so $k$
is odd and the clique number is $(k+1)/2$.)

\subsubsection{The prime $2$}

We can write $G=A\times B$, where $A$ has odd order and $B$ has order a power
of $2$. Suppose that $B$ is the product of $r$ cyclic groups, of which $q$
have order $2$. Then $2G=A\times2B$, where $B$ is the product of $r-q$ cyclic
groups, We have $G/2G$ elementary abelian of order $2^r$; so every coset
$2G+x$ satisfies $2x\in2G$, and there are $2^q$ cosets containing
involutions. Thus $m_1=0$, $m_2=2^r-2^q$, and $m_3=2^q$. Moreover,
$S(2G)$ is an elementary abelian of order $2^{r-q}$, so $S(2G)=2G$ if and only if
$G$ is the product of cyclic groups of orders $2$ and $4$ only.

So $\Gamma^+(G,2G)$ has clique number $k=|2G|=|G|/2^r$, and independence number
$2^r$; and $\Gamma(G,2G)$ has clique number $(k+2^{r-q})/2$, and independence
number $2^{r+1}$ unless $G$ is a product of cyclic groups of orders $2$
and~$4$, in which case the independence number is $2^{r+1}-2^q$.

\section{Spectrum}

Since the graphs $\Gamma^+_{G,H}$ and $\Gamma_{G,H}$ are disconnected if
$H<G,$ we can compute the spectrum by considering the components separately.
Theorem~\ref{t:main} gives the number of components of each type. As before
we let $|H|=k$ and $|G:H|=m$.

For $\Gamma^+_{G,H}$, the situation is simple: the components are either
complete bipartite $K_{k,k}$ (with eigenvalues $k$ and $-k$ each with
multiplicity~$1$, and $0$ with multiplicity $2k-2$) or complete $K_k$ (with
eigenvalues $k-1$ with multiplicity~$1$ and $-1$ with multiplicity $k-1$).

For $\Gamma_{G,H}$ things are a little more complicated. There are three types
of components:
\begin{enumerate}
\item $K_{k,k}$ minus a perfect matching. This is a distance-regular graph;
its eigenvalues are $k-1$ and $-(k-1)$, each with multiplicity~$1$, and
$1$ and $-1$, each with multiplicity $k-1$.
\item $K_k$ minus a perfect matching.
\item $K_k$ minus a partial matching covering $k-s(H)$ vertices.
\end{enumerate}

Consider the graph $K_{a+b}$, with $a$ even, minus a matching covering $a$
vertices. (We assume that $a>0$ and $a+b\ge4$ to avoid trivial cases.)
This graph occurs in various contexts. For example, it is a Tur\'an
graph, maximizing the number of edges in a graph containing no complete 
subgraph of order $(a/2)+b+1$. More relevant here, it is a \textit{generalized
line graph} in the sense of Hoffman; if $a>2$, its smallest eigenvalue is $-2$,
with multiplicity $(a/2)-1$ (see \cite{cgss}). For $a,b>0$, its eigenvalues
are as follows:
\begin{enumerate}
\item the roots of the quadratic $x^2-(a+b-3)x-(a+2b-2)$, each with
multiplicity $1$;
\item $0$, with multiplicity $a/2$;
\item $-2$, with multiplicity $(a/2)-1$;
\item $-1$, with multiplicity $b-1$.
\end{enumerate}

These (and the corresponding eigenvectors) can be calculated, using the fact that the partition into the vertices covered and uncovered by the matching is equitable, in the sense of Godsil and Royle~\cite{gr}. In detail: let the adjacency matrix have the form $A(\Gamma)=\displaystyle{\begin{pmatrix}A&J\\J&B\end{pmatrix}}$, where $A$ is the adjacency matrix of $K_a$ minus a matching, $B$ the adjacency matrix of $K_b$, and $J$ an all -$1$ matrix of appropriate size. Then the subspace of $\mathbb{R}^{a+b}$ spanned by $(1^a,0^b)$ and $(0^a,1^b)$ is $A(\Gamma)$-invariant, and the restriction of $(\Gamma)$ to this subspace is $\displaystyle{\begin{pmatrix}a-2&b\\a&b-1\\\end{pmatrix}}$, with eigenvalues as in (a). The orthogonal subspace is spanned by three types of vectors: those with zero entries in the last $b$ coordinates, and where entries in positions corresponding to the ends of (deleted) matching edges are negatives of each other; those with zero entries in the last $b$ coordinates, and with entries
in positions corresponding to the ends of matching edges are equal and all entries sum to $0$; and those with zero entries in the first $a$ coordinates, and with the remaining entries summing to zero. These are seen to be eigenvectors with the eigenvalues and multiplicities given in (b)--(d).

\section{Domination}
The \textit{domination number} of a graph is the least cardinality of a set of vertices for which the union of their closed neighbourhoods is the whole vertex set. The domination number of $K_k$ is clearly $1$, and that of $K_{k,k}$ is $2$. So the domination number of $\Gamma^+_{G,H}$ is $m=|G:H|$. Similarly the domination number of $K_{k,k}$ minus a perfect matching is $2$ (two vertices on one of the deleted edges form a dominating set), while the domination number of $K_k$ minus a matching is $1$ if the matching is not a perfect matching and $2$ if it is. So the domination number of $\Gamma_{G,H}$
is $m_1+2m_2+m_3=|G:H|+m_2$.

The complement of a disconnected graph $\Gamma$ with no isolated vertices has domination number $2$: two vertices in different components of $\Gamma$ form a dominating set in $\overline{\Gamma}$. By Proposition~\ref{p:conn}, provided that $H<G$, the domination numbers of $\overline{\Gamma}^+_{G,H}$ and
$\overline{\Gamma}_{G,H}$ are both $2$.

\section{Reconstructing the group}
In studies of graphs defined on groups, one topic which has been considered is the extent to which the graph determines the group. For example, Solomon and Woldar~\cite{SW} showed that the \emph{commuting graph} (with $x$ joined to $y$ if $xy=yx$) of a finite simple group determines the group.

In our situation, things are a bit different: the graphs only determine certain parameters of the pair $(G,H)$. We note that, in either case, the number of vertices of the graph is equal to $|G|$. The connected components have sizes $k$ and $2k$, and the number $k$ occurs at least once (for the subgroup $H$); so we can also determine $|H|$.

Suppose that we are given $\Gamma^+_{G,H}$. Then the number of components of size $k$ is equal to $s(G/H)$. So this is all the information we get: two pairs $(G_1,H_1)$ and $(G_2,H_2)$ with $s(G_1/H_1)=s(G_2/H_2)$ have isomorphic subgroup sum graphs.

For the graph $\Gamma_{G,H}$, we obtain the above information, and also the numbers $s(H)$ and $s(G)$. For the components of size $k$ are either complete minus a perfect matching of size $k/2$ ($m_2$ of these) or complete minus a matching of size $(k-s(H))/2$ ($m_3$ of these). Since $s(H)\ne0$, we can
distinguish the two types, and hence find $m_2$ and $m_3$; and from a Type~3 coset we can recover $s(H)$, and hence $s(G)=m_3s(H)$.

\section{A generalization}
Let $H$ and $K$ be subgroups of the abelian group $G$ with $H\neq K.$ We can extend our previous definition by defining the \textit{generalized subgroup sum graph} $\Gamma_{G;H,K}$ to be the graph with vertex set $G$ in which $x$ and $y$ are joined if $x+y\in H$ but $x+y\notin K$.

We lose no generality by assuming that $K\le H$, since with the definition just given, $\Gamma_{G;H,K}=\Gamma_{G;H,H\cap K}$. We shall always make this assumption. Now our previous subgroup sum graph $\Gamma_{G,H}$ is $\Gamma_{G;H,\{0\}}$. Moreover, $\Gamma_{G;G,H}$ is the complement of the extended subgroup sum graph $\Gamma^+_{G,H}$. We have not investigated these graphs further.

\section*{Acknowledgments} The research work T. Tamizh Chelvam is supported by CSIR Emeritus Scientist Scheme (No.21 (1123)/20/EMR-II) of  Council of Scientific and Industrial Research, Government of India.


\begin{thebibliography}{99} 
\bibitem{MK} M. Afkhami, K. Khashyarmanesh, K. Nafar, Generalized Cayley graphs associated to commutative rings, \textit{Linear Algebra Appl.} 437 (2012),  1040--1049.

\bibitem{MT} B. Alspach,   T. D. Parsons,  Isomorphism of Circulant Graphs and Digraphs, {\it Discrete Mathe.} 25 (1979), 97--108.

\bibitem{BR} B. Biswas, R. Sen Gupta, On the connectedness of square element graphs over arbitrary rings, {\it South East Asian Bull. Math.} 43(2)  (2019), 153--164.

\bibitem{RK} B. Biswas, R. Sen Gupta, M. K. Sen, S. Kar,  Some properties of square element graphs over semigroups, {\it AKCE International Journal of Graphs and Combinatorics} 17(1) (2020),  118--130. 

\bibitem{CS} P. J. Cameron and S. Ghosh, The power graph of a finite group, {\it Discrete Math.} 311 (2011), 1220--1222.

\bibitem{cgss}
P. J. Cameron, J.-M. Goethals, J. J. Seidel and E. E. Shult, Line graphs, root systems and elliptic geometry, \textit{J. Algebra} \textbf{43} (1976), 305--327.

\bibitem{IS} I. Chakrabarty, S. Ghosh and M. K. Sen, Undirected power graphs of semi groups, {\it Semigroup Forum} 78 (2009), 410--426.

\bibitem{EP} E. Dobson, P. Spiga and  G. Verret,  Cayley graphs on abelian groups, {\it Combinatorica}  36 (4) (2016), 371--393.
 
\bibitem{DEJ03} I. J. Dejter, O. Serra,  Efficient dominating sets in Cayley graphs, {\it Discrete Appl. Math.} 129 (2003), 319--328

\bibitem{Dilworth}
Robert P. Dilworth, A decomposition theorem for partially ordered sets, \textit{Ann. Math.} \textbf{51} (1950), 161--166.

\bibitem{gr}
Chris Godsil and Gordon Royle, \textit{Algebraic Graph Theory}, Springer,
New York, 2001.

\bibitem{QH} Q. Huang, A classification of circulant DCI(CI)- digraphs of 2-power order,  {\it Discrete Math.} 265 (2003), 71-- 84.

\bibitem{MH} M. Jixiang and  H. Qiongxiang,  Isomorphisms of circulant diagaphs, {\it Appl. Math. J. Chinese Univ. Ser. A },  9(1994), 405--409.

\bibitem{WT} W. Klotz and T. Sander, Integral Cayley graphs over abelian groups, {\it Electron. J. Combin.} 17 (2010), \# 81.

\bibitem{AV1} A. V. Kelarev and  S. J. Quinn, A combinatorial property and power graph of semigroups, {\it  Commentationes Mathematicae Universitatis Carolinae}, 45(1), (2004), 1--7.

\bibitem{LAK99} S. Lakshmivarahan  and S. K. Dhall, Ring, torus, hypercube architectures algorithms for parallel computing, {\it Parallel Computing},  25 (1999), 1877--1906. 
 
\bibitem{LEE01} J. Lee, Independent perfect domination sets in Cayley graphs, \textit{J. Graph Theory} 37(4) (2001), 213--219.

\bibitem{Lovasz}
L\'aszlo Lovász,  Normal hypergraphs and the perfect graph conjecture, \textit{Discrete Mathematics} \textbf{2} (1972), 253--267.

\bibitem{RPTT19} R. Raveendra Prathap and T. Tamizh Chelvam,  Complement graph of the square graph of finite abelian groups, {\it Houston Journal of Mathematics.} 46(4), (2020),  845--857.

\bibitem{RPTT20} R. Raveendra Prathap and T. Tamizh Chelvam, The cubic power graph of finite abelian groups, {\it AKCE International journal of graphs and combinatorics.} 18(1), (2021), 16--24.

\bibitem{RM} R. Sen Gupta and M. K. Sen, The square element graph over a finite commutative ring, {\it South East Asian Bull. Math.} 39 (3) (2015), 407--428.

\bibitem{SS} R. Sen Gupta, M. K. Sen,  The square element graph over a ring, {\it Southeast Asian Bull. Math.} 41 (5) (2017) 663--682.

\bibitem{SW}
R. M. Solomon and A. J. Woldar, Simple groups are characterized by their non-commuting graphs,
\textit{J. Group Theory} \textbf{16} (2013), 793--824.

\bibitem{T20196}T. Tamizh Chelvam and  S. Anukumar Kathirvel,  Generalized unit and unitary Cayley graphs of finite rings, {\it J. Algebra Appl.},  18(1) (2019), \# 1950006 [21 pages].
 	
\bibitem{T20194} T. Tamizh Chelvam,  S. Anukumar Kathirvel and M. Balamurugan, Domination in generalized unit and unitary Cayley graphs of finite rings, {\it Indian J. Pure Appl. Math.}  51(2) (2020), 533--556.
  
\bibitem{TTGKWC} T. Tamizh Chelvam, G. Kalaimurugan and Well Y. Chou, Signed star domination number of Cayley graphs, {\it Discrete Math. Algorithms Appl.} 4(2) (2012), \# 1250017 [10 pages].

\bibitem{TTSMU} T.	Tamizh Chelvam and S.  Mutharasu, Subgroups as efficient dominating set in Cayley Graphs, {\it Discrete Appl. Math.} 161(9)  (2013), 1187--1190.
 
\bibitem{TTSMU2} T.	Tamizh Chelvam and S. Mutharasu, Efficient open domination in Cayley graphs, {\it Appl. Math. Lett.} 25 (2012), 1560--1564.
 
\bibitem{TTSR1} T. Tamizh Chelvam and S. Raja,  Integral circulant graphs with four distinct eigenvalues, {\it Discrete Math. Algorithms Appl.} 10(5) (2018),  \# 1850062 [10 pages].
        
\bibitem{TTIR4}	T. Tamizh Chelvam and  I. Rani,  Independent Domination number of a Cayley Graph on $\mathbb{Z}_n$, {\it J. Combin. Math. Combin. Comput.} 69(2009), 251--255.
  
\bibitem{TTMS4} T. Tamizh Chelvam and M. Sivagami,  Structure and substructure connectivity of Cayley graphs, {\it Arab Journal of Mathematical Science} 27(1), (2021), 94--103. 

\bibitem{TTSRIG} T.	Tamizh Chelvam, S. Raja and I. Gutman, Strongly regular integral Circulant graphs and their energies, {\it Bulletin of International Mathematical Virtual Institute} 2 (2012), 9--16. 
 
\bibitem{TM} T. Tamizh Chelvam and M. Sattanathan,  Power graph of finite abelian groups, {\it Algebra and Discrete Mathematics} 16(1)  (2013), 33--41.
  
\bibitem{EY} E. Vatandoost and Y. Golkhandypour,  Integral Cayley graph on some of abelian groups, {TJEAS Journal} 3(6) (2013), 489--492.
\end{thebibliography}
\end{document}